\documentclass[10pt]{amsart}
\pdfoutput=1
\theoremstyle{plain}
\usepackage{amsmath, amsfonts, amsthm, amssymb,amscd,bbold}
\usepackage{graphics}
\usepackage{color}
\usepackage{epsfig}
\usepackage{tikz}
\usepackage{tikz-cd}
\usetikzlibrary{arrows}
\usepackage[all]{xy}
\usepackage{todonotes}
\graphicspath{ {Figures/} }

\theoremstyle{plain}
\newtheorem{theorem}{Theorem}
\newtheorem{lemma}{Lemma}

\newtheorem{question}{Question}

\theoremstyle{definition}

\newtheorem{remark}{Remark}

\newcommand{\Szabo}{{Szab{\'o}} }

\newcommand{\R}{\ensuremath{\mathbb{R}}}
\newcommand{\Z}{\ensuremath{\mathbb{Z}}}
\newcommand{\Ztwo}{\ensuremath{\mathbb{Z}_2}}

\newcommand{\C}{\ensuremath{\mathbb{C}}}

\newcommand{\OO}{\ensuremath{\mathbb{O}}}
\newcommand{\XX}{\ensuremath{\mathbb{X}}}

\newcommand{\cP}{\ensuremath{\mathcal{P}}}

\newcommand{\Wedge}{\ensuremath{\Lambda}}

\newcommand{\Hom}{{\rm Hom}}

\newcommand{\SKh}{\ensuremath{\mbox{AKh}_{odd}}}
\newcommand{\CKh}{\ensuremath{\mbox{CKh}_{odd}}}
\newcommand{\ACKh}{\ensuremath{\mbox{ACKh}_{odd}}}

\newcommand{\sltwo}{\ensuremath{\mathfrak{sl}(2)}}

\newcommand{\gloneone}{\ensuremath{\mathfrak{gl}(1|1)}}
\newcommand{\gl}{\ensuremath{\mathfrak{gl}}}

\newcommand{\cobe}{\mathcal{C}ob^3_{\ell}(\emptyset)}
\newcommand{\cobchrono}{\Bbbk\!\operatorname{\mathbf{ChCob}}_{/\ell}(0)}

\newcommand{\End}{\mbox{End}}

\newcommand{\id}{\mbox{id}}

\DeclareFontFamily{U}{MnSymbolC}{}
\DeclareSymbolFont{MnSyC}{U}{MnSymbolC}{m}{n}
\DeclareFontShape{U}{MnSymbolC}{m}{n}{
    <-6>  MnSymbolC5
   <6-7>  MnSymbolC6
   <7-8>  MnSymbolC7
   <8-9>  MnSymbolC8
   <9-10> MnSymbolC9
  <10-12> MnSymbolC10
  <12->   MnSymbolC12}{}
\DeclareMathSymbol{\intprod}{\mathbin}{MnSyC}{'270}
\DeclareMathSymbol{\intprodrev}{\mathbin}{MnSyC}{183}

\author{J. Elisenda Grigsby}
\thanks{JEG was partially supported by NSF CAREER award DMS-1151671 and a grant from the Simons Foundation/SFARI (396324, JEG)}
\address{Boston College; Department of Mathematics; 5th floor Maloney; Chestnut Hill, MA 02467}
\email{grigsbyj@bc.edu}

\author{Stephan M. Wehrli}
\thanks{SMW was partially supported by NSF grant DMS-1111680}
\address{Syracuse University; Department of Mathematics; 215 Carnegie; Syracuse, NY 13244}
\email{smwehrli@syr.edu}

\title[An action of $\protect\gloneone$ on odd annular Khovanov homology]{An action of $\ensuremath{\boldsymbol{\mathfrak{gl}(1|1)}}$ on odd annular Khovanov homology}

\begin{document}
\bibliographystyle{plain}
\maketitle
\begin{abstract}
We define an annular version of odd Khovanov homology and prove that it
carries an action of the Lie superalgebra $\gloneone$ which is preserved under annular Reidemeister moves.\end{abstract}

\section{Introduction}
In \cite{K}, Khovanov defined a link invariant {\em categorifying} the Jones polynomial. That is, he constructed a bigraded homology theory of links whose graded Euler characteristic is the Jones polynomial. Recalling \cite{Witten, ReshTur} that the Jones polynomial has an interpretation involving representations of the quantum group $U_q(\sltwo)$, it is perhaps not surprising that an {\em annular} version of Khovanov homology defined in \cite{APS} and further studied in \cite{LRoberts} (see also \cite{AnnularLinks}) carries an action of the Lie algebra $\sltwo$ \cite{SchurWeyl, QR}. Moreover, although annular Khovanov homology is {\em not} a link invariant (it is well-defined only up to isotopy in the complement of a standardly-imbedded unknot in $S^3$), the algebraic features of the Khovanov complex that have yielded the most geometric/topological information (e.g. \cite{Plam, Rasmussen_Slice}) have coincided with key $\sltwo$--representation-theoretic features of the annular Khovanov complex (cf. \cite[Prop. 1]{dt}).

In \cite{ORS}, Ozsv{\'a}th-Szab{\'o}-Rasmussen defined an {\em odd} version of Khovanov homology. When taken with $\Ztwo$ coefficients, their construction agrees with Khovanov's original construction. Hence, odd Khovanov homology can be viewed as an alternative integral lift of Khovanov homology. 

Our aim in the present work is to define a natural {\em annular} version of {\em odd} Khovanov homology (Subsection \ref{subs:oddAKh}) and show that it carries a well-defined action, not of the Lie algebra $\sltwo$, but of the Lie superalgebra $\gl(1|1)$ (Theorems \ref{theorem:invariance} and \ref{thm:homequiv}). We will define this $\gl(1|1)$ action explicitly on chain level, using two different descriptions of the odd annular Khovanov complex (Subsections~\ref{subs:definition} and \ref{subs:alternative}).

In a follow-up paper with Casey Necheles, the second author extends results of Russell~\cite{Russell} to the odd setting by relating an annular version of Putyra's chronological cobordism category~\cite{PutyraChrono} to a dotted version of the odd Temperley-Lieb category (defined as in \cite{BrundanEllis}) at $\delta=0$. After setting dots equal to zero, the latter category becomes equivalent to a (non-full) subcategory of the category of $\gl(1|1)$ representations. As a consequence, one obtains a natural interpretation of the $\gl(1|1)$ action on odd annular Khovanov homology.

In a different direction, noting that: 
\begin{itemize}
	\item the bordered Heegaard-Floer tangle invariant defined by Petkova-V{\'e}rtesi \cite{pv} carries a categorical action of $U_q(\gl(1|1))$ \cite{epv},
	\item on a decategorified level, the bordered theory for knot Floer homology defined by Ozsv{\'a}th-Szab{\'o} \cite{oszKauffman} carries an action of $U_q(\gl(1|1))$ \cite{ManionDecat},
	\item conjecturally, there is an Ozsv{\'a}th-Szab{\'o} spectral sequence relating odd Khovanov homology of (the mirror of) a link to the Heegaard-Floer homology of the manifold obtained as the connected sum of the double-branched cover of $L$ with $S^1 \times S^2$ (cf. \cite{Sz, Be}),
\end{itemize}

it is natural to ask the following:
\begin{question} Let $L \subseteq Y$ be a link in a $3$--manifold satisfying either:
\begin{enumerate}
	\item $Y = S^3$ and $L= L_0 \cup L'$ where $L_0$ is an unknot and $L'$ is non-empty, or
	\item $Y$ is the double-branched cover of a knot $K \subseteq S^3$, and $L = p^{-1}(U)$ is the preimage of an unknot $U$ in $S^3 - K$.
\end{enumerate}
Does $\widehat{HFK}(Y,L)$, the knot Floer homology of $L$ in $Y$, carry an action of $\gl(1|1)$? In the latter case, how does this $\gl(1|1)$ action relate to the $\gl(1|1)$ action on $\SKh(K \subseteq S^3 - N(U))$?
\end{question}

For a link $L=L_0 \cup L'$ which is realized as the closure, $L'$, of a tangle $T$, linked by the tangle axis, $L_0$, Petkova-V{\'e}rtesi~\cite{pv1} showed that the knot Floer homology of $L$ can be identified with the Hochschild homology of the tangle Floer homology of $T$. In the case where $T$ is a tangle in $\mathbb{R}^2\times I$, the existence of a $\gl(1|1)$ action on $\widehat{HFK}(L)$ could therefore be established by showing that the
categorical $\gl(1|1)$ action described in \cite{epv} induces an action on Hochschild homology.

In this context, it is worth noting that the even version of annular Khovanov homology has been identified (see \cite{HochHom} for a special case and \cite{Beliakova-Putyra-Wehrli} for the general case) with the Hochschild homology of the Chen-Khovanov tangle invariant, which categorifies the Reshetikhin-Turaev tangle invariant associated to the fundamental representation of $U_q(\mathfrak{sl}(2))$.

\subsection{Acknowledgements} The authors would like to thank John Baldwin, Tony Licata, Robert Lipshitz, and Ina Petkova for interesting discussions.

\section{Preliminaries on $\gl(1|1)$ representations}\label{section:representations}
In this section, we will review basic facts of the representation theory of the Lie superalgebra $\gl(1|1)$. We will assume throughout that we are working over $\C$. By a vector superspace, we will mean a vector space $V$ endowed with a $\Ztwo$-grading $V=V_{\bar{0}}\oplus V_{\bar{1}}$. We will refer to this grading as the {\em supergrading} on $V$, and we will use the notation $|v|$ to denote the superdegree of a homogeneous element $v\in V$.

For $n\in\Z$, we will denote by $\langle n\rangle$ the shift functor which shifts the superdegree on a vector superspace by the image $\bar{n}$ of $n$ under the quotient map $\Z\rightarrow\Z_2$. Thus if $V$ is a vector superspace, then $V\langle n\rangle$ is the vector superspace with $(V\langle n\rangle)_{\bar{i}}=V_{\bar{i}+\bar{n}}$ for $\bar{i}\in\Ztwo$.

\subsection{Representations of Lie superalgebras}\label{subs:representations} Recall that a {\em Lie superalgebra} is a vector superspace $\mathfrak{g}=\mathfrak{g}_{\bar{0}}\oplus\mathfrak{g}_{\bar{1}}$ endowed with a bilinear {\em Lie superbracket} $[-,-]_s\colon\mathfrak{g}\times\mathfrak{g}\rightarrow\mathfrak{g}$ satisfying
\begin{enumerate}
\item
$|[x,y]_s|=|x|+|y|$,
\item
$[x,y]_s=-(-1)^{|x||y|}[y,x]_s$,
\item
$[x,[y,z]_s]_s=[[x,y]_s,z]_s+(-1)^{|x||y|}[y,[x,z]_s]_s$,
\end{enumerate}
for all homogeneous $x,y,z\in\mathfrak{g}$.

If $V$ is a vector superspace, then $\gl(V)$ denotes the Lie superalgebra whose underlying vector superspace is the space $\End(V)$ of all linear endomorphisms of $V$, and whose Lie superbracket is given by the supercommutator $
[x,y]_s:=x\circ y - (-1)^{|x||y|}y\circ x$ for all homogeneous $x,y\in\End(V)$. Here, it is understood that an endomorphism of $V$ has superdegree $\bar{0}$ if it preserves the supergrading on $V$, and superdegree $\bar{1}$ if it reverses the supergrading on $V$.
 
A {\em homomorphism} between two Lie superalgebras $\mathfrak{g}$ and $\mathfrak{g}^\prime$ is a linear map from $\mathfrak{g}$ to $\mathfrak{g}^\prime$ which preserves both the supergrading and the Lie superbracket. A {\em representation} of a Lie superalgebra $\mathfrak{g}$ is a vector superspace $V$ together with a homomorphism $\rho_V\colon\mathfrak{g}\rightarrow\gl(V)$. As with ordinary Lie algebras, the map $\rho_V$ is sometimes called the {\em action} of $\mathfrak{g}$ on $V$.

Let $V$ and $W$ be two representations of a Lie superalgebra $\mathfrak{g}$. Then $V\langle n\rangle$ is a representation of $\mathfrak{g}$ with $\rho_{V\langle n\rangle}=\rho_V$, and the dual space $V^*=\Hom(V,\C)$ is a representation of $\mathfrak{g}$ with $(V^*)_{\bar{i}}=\Hom(V_{\bar{i}},\C)$ for $\bar{i}\in\Z_2$ and
\[
\rho_{V^*}(x)(\varphi)=-(-1)^{|x||\varphi|}(\rho_V(x))^*(\varphi)
\]
for all homogeneous $x\in\mathfrak{g}$ and $\varphi\in V^*$. Moreover, the tensor product $V\otimes W$ is a representation of $\mathfrak{g}$ with $|v\otimes w|=|v|+|w|$ and
\[
\rho_{V\otimes W}(x)(v\otimes w) = (-1)^{|w||x|}(\rho_V(x)(v))\otimes w + v\otimes(\rho_W(x)(w))
\]
for all homogeneous $x\in\mathfrak{g}$, $v\in V$, and $w\in W$.

\begin{remark} In the literature, the action of $\mathfrak{g}$ on $V\otimes W$ is more commonly defined by
\[
\rho_{V\otimes W}(x)(v\otimes w) = (\rho_V(x)(v))\otimes w + (-1)^{|x||v|}v\otimes(\rho_W(x)(w))
\]
for all homogeneous $x\in\mathfrak{g}$, $v\in V$, and $w\in W$. It is easy to check that our definition yields an isomorphic representation, where the isomorphism is given by the endomorphism of $V\otimes W$ which sends a homogeneous element $v\otimes w$ to the element $(-1)^{|v||w|}v\otimes w$.
\end{remark}

\begin{remark}\label{remark:twist} The representations $V\otimes W$ and $W\otimes V$ are isomorphic where the isomorphism is given by the linear map $\tau\colon V\otimes W\rightarrow W\otimes V$ which sends a homogeneous element $v\otimes w$ to the element $(-1)^{|v||w|}w\otimes v$. We will henceforth call this map the {\em twist map}.
\end{remark}

Given two homogeneous linear maps $f\colon V\rightarrow V^\prime$ and $g\colon W\rightarrow W^\prime$ between vector superspaces, let $f\otimes g\colon V\otimes W\rightarrow V^\prime\otimes W^\prime$ denote the homogeneous linear map defined by
\[
(f\otimes g)(v\otimes w):=(-1)^{|g||v|}f(v)\otimes g(w)
\]
for all homogeneous $v\in V$ and $w\in W$. Using this definition, we have:

\begin{lemma}\label{lemma:tensor} If $\mathfrak{g}$ is a Lie superalgebra and $f\colon V\rightarrow V^\prime$ and $g\colon W\rightarrow W^\prime$ are homogeneous linear maps between $\mathfrak{g}$ representations which intertwine the actions of $\mathfrak{g}$, then the tensor product $f\otimes g\colon V\otimes W\rightarrow V^\prime\otimes W^\prime$ defined as above also intertwines the actions of $\mathfrak{g}$.
\end{lemma}

\begin{proof} Let $x\in\mathfrak{g}$, $v\in V$, and $w\in W$ be homogeneous. Then
\begin{align*}
\rho_{V^\prime\otimes W^\prime}&(x)\Bigl[(f\otimes g)(v\otimes w)\Bigr]=(-1)^{|g||v|}\rho_{V^\prime\otimes W^\prime}(x)\Bigl[f(v)\otimes g(w)\Bigr]\\
&=(-1)^{|g||v|}\Bigl[(-1)^{(|g|+|w|)|x|}\rho_{V^\prime}(x)(f(v))\otimes g(w)+f(v)\otimes \rho_{W^\prime}(x)(g(w))\Bigr]\\
&=(-1)^{|g||v|}\Bigl[(-1)^{(|g|+|w|)|x|}f(\rho_{V}(x)(v))\otimes g(w)+f(v)\otimes g(\rho_{W}(x)(w))\Bigr]\\
&=(-1)^{|g||v|+|g||x|+|w||x|}f(\rho_V(x)(v))\otimes g(w)+(-1)^{|g||v|}f(v)\otimes g(\rho_W(x)(w))\\
&=(-1)^{|w||x|}(-1)^{|g|(|x|+|v|)}f(\rho_V(x)(v))\otimes g(w)+(-1)^{|g||v|}f(v)\otimes g(\rho_W(x)(w))\\
&=(-1)^{|w||x|}(f\otimes g)(\rho_V(x)(v)\otimes w)+(f\otimes g)(v\otimes\rho_W(x)(w))\\
&=(f\otimes g)\Bigl[(-1)^{|w||x|}\rho_V(x)(v)\otimes w+v\otimes\rho_W(x)(w)\Bigr]\\
&=(f\otimes g)\Bigl[\rho_{V\otimes W}(x)(v\otimes w)\Bigr].
\end{align*}
Hence $f\otimes g$ intertwines the maps $\rho_{V\otimes W}(x)$ and $\rho_{V^\prime\otimes W^\prime}(x)$, which proves the lemma.
\end{proof}

For a vector superspace $V$, let $\Phi_V$ denote the linear involution $\Phi_V\colon V\rightarrow V$ defined by $\Phi_V(v)=(-1)^{|v|}$ for every homogeneous element $v\in V$. The following lemma describes how the grading shift functor $\langle 1\rangle$ interacts with duals and tensor products of representations.

\begin{lemma} Let $V$ and $W$ be two representations of a Lie superalgebra $\mathfrak{g}$. Then \[(V\langle 1\rangle)^*\cong V^*\langle 1\rangle,\] where the isomorphism is given by the map $\Phi_{V^*}\colon V^*\rightarrow V^*$, and
\[
V\langle 1\rangle\otimes W\cong (V\otimes W)\langle 1\rangle\cong V\otimes W\langle 1\rangle,
\]
where the first isomorphism is given by the identity map of $V\otimes W$ and the second isomorphism is given by the map $\Phi_V\otimes\id_W$.
\end{lemma}

\begin{proof} Let $x\in\mathfrak{g}$ and $\varphi\in V^*$ be homogeneous. Then the definition of $\Phi_{V^*}$ implies
\[
(\Phi_{V^*}\circ\rho_{V^*}(x))(\varphi) = (-1)^{|x|+|\varphi|}(\rho_{V^*}(x))(\varphi) = (-1)^{|x|}(\rho_{V^*}(x)\circ\Phi_{V^*})(\varphi),
\]
where the first equation follows because $\rho_{V^*}(x)(\varphi)$ has superdegree $|x|+|\varphi|$.
Because $\rho_{V^*}=\rho_{V^*\langle 1\rangle}$, the left-most term in the above sequence of equations can be identified with $(\Phi_{V^*}\circ\rho_{V^*\langle 1\rangle}(x))(\varphi)$, and because
\begin{align*}
(-1)^{|x|}\rho_{V^*}(x)(\varphi)&= -(-1)^{|x|}(-1)^{|x||\varphi|}(\rho_V(x))^*(\varphi)\\
&=-(-1)^{|x|(|\varphi|+1)}(\rho_{V\langle 1\rangle}(x))^*(\varphi)\\
&=\rho_{(V\langle 1\rangle)^*}(x)(\varphi),
\end{align*}
the right-most term can be identified with $(\rho_{(V\langle 1\rangle)^*}(x)\circ\Phi_{V^*})(\varphi)$. Thus we have
\[
\Phi_{V^*}\circ\rho_{V^*\langle 1\rangle}(x)=\rho_{(V\langle 1\rangle)^*}(x)\circ\Phi_{V^*},
\]
and hence $\Phi_{V^*}$ is an isomorphism between $V^*\langle 1\rangle$ and $(V\langle 1\rangle)^*$.

The claim that the identity map of $V\otimes W$ is an isomorphism between $V\langle 1\rangle\otimes W$ and $(V\otimes W)\langle 1\rangle$ follows because $\rho_{V\langle 1\rangle\otimes W}=\rho_{V\otimes W}=\rho_{(V\otimes W)\langle 1\rangle}$, by our definition of the tensor product of two representations.

Finally, by Remark~\ref{remark:twist}, we have a sequence of isomorphisms
\[
V\otimes W\langle 1\rangle\stackrel{\tau}{\longrightarrow}W\langle 1\rangle\otimes V\stackrel{id}{\longrightarrow}
(W\otimes V)\langle 1\rangle\stackrel{\tau}{\longrightarrow}(V\otimes W)\langle 1\rangle.
\]
This sequence takes a homogeneous element $v\otimes w\in V\otimes W$ to
\[
(-1)^{|v|(|w|+1)}(-1)^{|w||v|}v\otimes w=(-1)^{|v|}v\otimes w=(\Phi_V\otimes\id_W)(v\otimes w),
\]
and hence $\Phi_V\otimes\id_W$ is an isomorphism between $V\otimes W\langle 1\rangle$ and $(V\otimes W)\langle 1\rangle$.
\end{proof}

\subsection{The Lie superalgebra $\mathbf{\mathfrak{gl}(1|1)}$ and some of its representations}
Let $\C^{1|1}$ denote the vector superspace \[\C^{1|1}:=\C v_0\oplus\C v_1\] spanned by two homogeneous elements $v_0$ and $v_1$ of superdegrees $\bar{0}$ and $\bar{1}$, respectively. The Lie superalgebra $\gl(1|1)$ is defined as the
space of linear endomorphisms $\gl(1|1)=\gl(\C^{1|1})=\End(\C^{1|1})$ with Lie superbracket given by the supercommutator, as described in the previous subsection. Explicitly, $\gl(1|1)$ is spanned by the following elements
\[
h_1=\begin{pmatrix}1&0\\0&0\end{pmatrix},\quad h_2=\begin{pmatrix}0&0\\0&1\end{pmatrix},\quad
e=\begin{pmatrix}0&1\\0&0\end{pmatrix},\quad f=\begin{pmatrix}0&0\\1&0\end{pmatrix},
\]
where $h_1$ and $h_2$ have superdegree $\bar{0}$ and $e$ and $f$ have superdegree $\bar{1}$. The Lie superbracket on $\gl(1|1)$ is given by
\[\arraycolsep=10pt
\begin{array}{ll}
[e,f]_s=h_1+h_2, &[e,e]_s=[f,f]_s=[h_i,h_j]_s=0,\\ \relax
[e,h_1]_s=-e, &[f,h_1]_s=f,\\ \relax
[e,h_2]_s=e, &[f,h_2]_s=-f,\\
\end{array}
\]
where $i,j\in\{1,2\}$.

Let $h_+:=h_1+h_2$ and $h_-:=h_1-h_2$. Then $h_+$ is central in $\gl(1|1)$, in the sense that $[h_+,x]_s=0$ for all $x\in\gl(1|1)$, and $h_-$ satisfies
\[
[e,h_-]_s=-2e,\quad [f,h_-]_s=2f,\quad [h_-,h_-]_s=0.
\]
Note that the elements $h_+$, $h_-$, $e$, and $f$ form a basis for $\gl(1|1)$ and that the Lie superbracket on $\gl(1|1)$ is completely characterized by the aforementioned properties of $h_+$ and $h_-$ and by the relations $[e,f]_s=h_+$ and $[e,e]_s=[f,f]_s=0$.

We will now describe a family of irreducible $\gl(1|1)$ representations $L_{(m,n)}$ parameterized by pairs of integers $(m,n)\in\Z^2$. It is not hard to see that every finite-dimensional irreducible $\gl(1|1)$ representation on which $h_1$ and $h_2$ act with integer eigenvalues is isomorphic to one of the representations in this family, up to a possible shift of the supergrading. See \cite{BM} and \cite{Sar}.

If $(m,n)\in\Z^2$ satisfies $m+n=0$, then $L_{(m,n)}$ is $1$-dimensional and supported in superdegree $\bar{0}$. The elements $(h_1,h_2)$ act by scalar multiplication by $(m,n)$ on this representation, and the elements $e$ and $f$ act by zero. Note that $L_{(0,0)}$ is a trivial representation, where trivial means that all generators of $\gl(1|1)$ act by zero.

If $(m,n)\in\Z^2$ satisfies $m+n\neq 0$, then $L_{(m,n)}$ is $2$-dimensional and spanned by two homogeneous vectors $v_+$ and $v_-$ of superdegrees $\bar{n}$ and $\bar{n}+\bar{1}$, respectively. In this case, the action of $\gl(1|1)$ relative to the basis $\{v_+,v_-\}$ is given by the following matrices:
\[
\begin{array}{clcclc}
\rho_{L_{(m,n)}}(h_1)&=&\begin{pmatrix}m&0\\0&m-1\end{pmatrix},\quad& \rho_{L_{(m,n)}}(h_2)&=&\begin{pmatrix}n&0\\0&n+1\end{pmatrix},\\[0.2in]
\rho_{L_{(m,n)}}(e)&=&\begin{pmatrix}0&m+n\\0&0\end{pmatrix},\quad& \rho_{L_{(m,n)}}(f)&=&\begin{pmatrix}
0&0\\1&0\end{pmatrix}.
\end{array}
\]

Note that for $m=1$ and $n=0$, the above matrices coincide with the matrices $h_1$, $h_2$, $e$, and $f$. Thus, the representation $L_{(1,0)}$ is equal to the fundamental representation $\C^{1|1}$ of $\gl(1|1)$, on which $\gl(1|1)$ acts by $\rho_{\C^{1|1}}=\id_{\gl(1|1)}$. In the remainder of this section, we will denote the fundamental representation $L_{(1,0)}=\C^{1|1}$ by $V$. The dual representation, $V^*$, can be seen to be isomorphic to the representation $L_{(0,-1)}$, where an isomorphism is given by the linear map which sends the basis $\{v_+^*,v_-^*\}$ of $V^*=L_{(1,0)}^*$ to the basis $\{v_-,v_+\}$ of $L_{(0,-1)}$.

In the next subsection, we will see that the representations $V^*\otimes V$ and $V\otimes V^*$ each contain a trivial $1$-dimensional subrepresentation and a trivial $1$-dimensional quotient representation, but no trivial direct summand. In particular, these representations are indecomposable but not irreducible. We will actually study the isomorphic (up to a grading shift) representations $V^*\langle 1\rangle\otimes V$ and $V\otimes (V^*\langle 1\rangle)$ since these representations will be needed later in this paper.

\subsection{$\mathbf{\mathfrak{gl}(1|1)}$ action on the representations $V^*\langle 1\rangle\otimes V$ and $V\otimes(V^*\langle 1\rangle)$}\label{subs:productaction} As before, let $V$ denote the fundamental representation of $\gl(1|1)$. For reasons that will become clear later, we will identify the vector superspaces underlying the representations $V^*\langle 1\rangle$ and $V$ by using the identifications $v_+^*=-v_-$ and $v_-^*=v_+$. It is not hard to see that, under these identifications, the action of $\gl(1|1)$ on $V^*\langle 1\rangle\otimes V$ is given as follows:
\[
\begin{tikzcd}
& v_-\otimes v_+-v_+\otimes v_- & \\
v_+\otimes v_+ \arrow{ur}{f=-1} \arrow[loop left]{l}{(h_1,h_2)=(1,-1)}& & v_-\otimes v_- \arrow{ul}[swap]{e=-1} \arrow[loop right]{r}{(h_1,h_2)=(-1,1)}\\
& v_+\otimes v_- + v_-\otimes v_+ \arrow{ul}{e=2} \arrow{ur}[swap]{f=-2} &
\end{tikzcd}
\]

For example, $e$ sends the vector $v_+\otimes v_- + v_-\otimes v_+$ to $2v_+\otimes v_+$ and $f$ sends this vector to $-2v_-\otimes v_-$. Likewise, $h_1$ and $h_2$ annihilate the vectors $v_+\otimes v_-$ and $v_-\otimes v_+$ and act on the vectors $v_+\otimes v_+$ and $v_-\otimes v_-$ by scalar multiplication by $(h_1,h_2)=(1,-1)$ and $(h_1,h_2)=(-1,1)$, respectively. Since there are no arrows ending at the vector $v_+\otimes v_-+v_-\otimes v_+$, this vector spans a trivial $1$-dimensional quotient representation, and since there are no arrows emanating from $v_-\otimes v_+-v_+\otimes v_-$, this vector spans a trivial $1$-dimensional subrepresentation. Let $p\colon V^*\langle 1\rangle\otimes V\rightarrow \C\langle 1\rangle$ and $i\colon\C\langle 1\rangle\rightarrow V^*\langle 1\rangle\otimes V$ denote the associated projection and inclusion maps, where $\C$ denotes the trivial $1$-dimensional representation $\C=L_{(0,0)}$. Explicitly, $p$ sends $v_+\otimes v_-$ and $v_-\otimes v_+$ to $1$ and $v_+\otimes v_+$ and $v_-\otimes v_-$ to zero, and $i$ sends $1$ to $v_-\otimes v_+-v_+\otimes v_-$. For later use, we also the introduce maps
\[
\tilde{p}\colon V^*\langle 1\rangle\otimes V\longrightarrow\C\oplus(\C\langle 1\rangle),\qquad\tilde{i}\colon \C\oplus(\C\langle 1\rangle)\longrightarrow V^*\langle 1\rangle\otimes V,
\]
defined as the compositions
\[
V^*\langle 1\rangle\otimes V\stackrel{p}{\longrightarrow}\C\langle 1\rangle\stackrel{i_2}{\longrightarrow}\C\oplus(\C\langle 1\rangle),\qquad
\C\oplus(\C\langle 1)\rangle\stackrel{p_1}{\longrightarrow}\C\langle 1\rangle\stackrel{i}{\longrightarrow}V^*\langle 1\rangle\otimes V,
\]
respectively, where $i_2\colon\C\langle 1\rangle\rightarrow\C\oplus(\C\langle 1\rangle)$ is the inclusion of $\C\langle 1\rangle$ into the second summand, and $p_1\colon \C\oplus(\C\langle 1\rangle)\rightarrow\C\langle 1\rangle$ is the projection onto the first summand, up to a shift of the supergrading. By construction, $\tilde{p}$ and $\tilde{i}$ intertwine the action of $\gl(1|1)$ and have superdegrees $\bar{0}$ and $\bar{1}$, respectively.

The action of $\gl(1|1)$ on $V\otimes (V^*\langle 1\rangle)$ is given by almost the same diagram as the one above, with the only difference being that the arrow labeled $f=-1$ is replaced by an arrow labeled $f=1$ and the arrow labeled $f=-2$ is replaced by an arrow labeled $f=2$. In particular, there are maps $\tilde{p}\colon V\otimes (V^*\langle 1\rangle)\rightarrow\C\oplus(\C\langle 1\rangle)$ and $ \tilde{i}\colon \C\oplus(\C\langle 1\rangle)\rightarrow V\otimes(V^*\langle 1\rangle)$ which are defined by the same formulas as before and which intertwine the action of $\gl(1|1)$.

\subsection{$\mathbf{\mathfrak{gl}(1|1)}$ actions on the exterior algebra of a vector space}\label{subs:exterior_action} We end this section with a natural construction of $\gl(1|1)$ actions on the exterior algebra of a vector space.

Let $U$ be a vector space equipped with a symmetric bilinear inner product $\langle-,-\rangle$ and let $U^\prime:=\Wedge^*(U)$ be its exterior algebra. We may regard $U^\prime$ as a vector superspace by collapsing the natural $\Z_{\geq 0}$-grading on the exterior algebra to a $\Z_2$-grading. Choose two vectors $a,b\in U$ and a constant $N\in\C$, and define a linear map $\rho_{U^\prime}\colon\gl(1|1)\rightarrow\gl(U^\prime)$ by
\[
\begin{array}{rcl}
\rho_{U^\prime}(h_+)(v)&=&\langle a,b\rangle v,\\
\rho_{U^\prime}(h_-)(v)&=&(N-2\ell)v,\\
\rho_{U^\prime}(e)(v)&=&a\intprod v,\\
\rho_{U^\prime}(f)(v)&=&b \wedge v,\\
\end{array}
\]
where $v\in U^\prime$ is an element of the form $v=u_1\wedge\ldots\wedge u_\ell$ for $u_j\in U$ and $\ell\geq 0$, and $\intprod$ is defined by
\[
u\intprod v=\sum_{j=1}^\ell(-1)^{j-1}\langle u,u_j\rangle u_1\wedge\ldots\wedge\widehat{u_j}\wedge\ldots\wedge u_\ell
\]
for $u\in U$ and $v\in U^\prime$ as before. Using this definition of $\intprod$, one can see that
\[
a\intprod(b\wedge v)+b\wedge(a\intprod v)=\langle a,b\rangle v
\]
for all $v\in U^\prime$. In particular, this implies that the map $\rho_{U^\prime}$ is compatible with the $\gl(1|1)$ relation $[e,f]_s=h_+$. Since $\rho_{U^\prime}(h_+)$ is given by multiplication by the constant $\langle a,b\rangle$, it is further clear that $\rho_{U^\prime}(h_+)$ is central in $\gl(U^\prime)$. Moreover, the map $v\mapsto a\intprod v$ (resp., $v\mapsto b\wedge v$) lowers (resp., raises) the number of factors in a wedge product by one, and together with the definition of $\rho_{U^\prime}(h_-)$, this implies that $\rho_{U^\prime}$ respects the $\gl(1|1)$ relations $[e,h_-]_s=-2e$ and $[f,h_-]_s=2f$. It is easy to see that $\rho_{U^\prime}$ is also compatible with the $\gl(1|1)$ relations $[e,e]_s=[f,f]_s=[h_-,h_-]_s=0$, and hence $\rho_{U^\prime}$ endows $U^\prime$ with a well-defined action of $\gl(1|1)$.

One can slightly modify the above definition of $\rho_{U^\prime}$ by replacing $b\wedge v$ by $v\wedge b$ and $a\intprod v$ by $v\intprodrev a$,  where $v\intprodrev a:= (-1)^{\ell-1} a\intprod v$ for $v\in U^\prime$ as before. This modified definition also yields a well-defined action of $\gl(1|1)$ on $U^\prime$.

\section{Odd annular Khovanov homology as a $\gl(1|1)$ module}
\subsection{Odd annular Khovanov homology} \label{subs:oddAKh}
Let $A$ be a closed, oriented annulus, $I = [0,1]$ the closed, oriented unit interval. Via the identification 
\[A \times I= \{(r,\theta,z)\,\,\vline\,\,r \in [1,2], \theta \in [0, 2\pi), z\in [0,1]\} \subset (S^3 = \R^3 \cup \infty) ,\] any link, $L \subset A \times I$, may naturally be viewed as a link in the complement of a standardly imbedded unknot, $(U = z$--axis $\cup \,\,\infty) \subset S^3$. Such an {\em annular link} $L \subset A \times I$ admits a diagram, $\cP(L) \subset A,$ obtained by projecting a generic isotopy class representative of $L$ onto $A \times \{1/2\}$.  From this diagram one can construct a triply-graded chain complex called the {\em annular} Khovanov complex associated to the annular link $L$, by using a version of Khovanov's original construction \cite{K} due to Asaeda-Przytycki-Sikora \cite{APS} and L. Roberts \cite{LRoberts} (see also \cite{AnnularLinks}). We now proceed to describe an odd version of the annular Khovanov complex, using the construction of Ozsv{\'a}th-Rasmussen-\Szabo in \cite{ORS}. 

Begin by decorating the diagram with an arrow (called an {\em orientation} in \cite{ORS}) at each crossing, as follows. Position the crossing so its overstrand connects the upper left (NW) to the lower right (SE) corner and draw an arrow on the crossing pointing either up or down. This choice will specify arrows for the two resolutions of the crossing, as follows. The arrow for the ``0" resolution will agree with the arrow at the crossing, and the arrow for the ``1" resolution will be rotated $90^\circ$ clockwise. See Figure \ref{fig:arrows}. 

\begin{figure}
\includegraphics[height=0.85in]{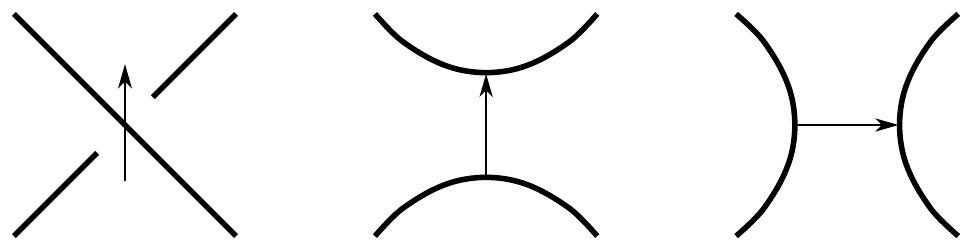}
\caption{Direction of the arrows in the ``0'' resolution and the ``1'' resolution.}
\label{fig:arrows}
\end{figure}

Now view the decorated diagram $\cP(L) \subset A$ instead as a diagram on $S^2 - \{\XX,\OO\}$, where $\XX$ (resp., $\OO$) are basepoints on $S^2$ corresponding to the inner (resp., outer) boundary circles of $A$. If we temporarily forget the data of $\XX$, we may view $\cP(L)$ as a diagram on $\R^2 = S^2 - \{\OO\}$ and form the ordinary bigraded odd Khovanov complex
\[\CKh(\cP(L)) = \bigoplus_{(i,j) \in \Z^2} \CKh^i(\cP(L);j)\] as described in \cite{ORS} and briefly recalled below.

Let $\mathcal{X}$ denote the set of crossings of $\cP(L)$. For each map $I: \mathcal{X} \rightarrow \{0,1\}$ one obtains an associated decorated imbedded $1$--manifold, $\cP^I(L) \subseteq S^2 - \{\OO\}$ obtained by resolving and decorating each crossing as specified by $I$.  Choosing an ordering of the $c$ crossings identifies these decorated complete resolutions with the vertices of a $c$--dimensional hypercube whose edges correspond to saddle cobordisms between decorated complete resolutions. 

Remembering the data of $\XX$, we now associate to each vertex of this {\em hypercube of decorated resolutions} a chain complex whose underlying $\Z^3$--graded vector space is defined as follows. For each $I: \mathcal{X} \rightarrow \{0,1\}$ let \[V(I) := \mbox{Span}_\C\{a_1, \ldots, a_n\}\] be the formal span of the components, $a_1, \ldots, a_n$, of $\cP^I(L)$. Then the $\C$ vector space we assign to the vertex $I$ is \[F(I) := \Wedge^*(V(I)).\] It will be convenient to note that $F(I)$ has a distinguished basis indexed by subsets $S \subseteq \{1, \ldots, n\}$. Given such a subset $S = \{i_1, \ldots, i_\ell\} \subseteq \{1, \ldots, n\}$ whose elements have been arranged in order ($i_1 < \ldots< i_\ell$), we will denote the associated basis element of $F(I)$ by \[a_S := a_{i_1} \wedge \ldots \wedge a_{i_\ell}.\] 

Now each vector space $F(I)$ is endowed with a homological ($i$) and quantum ($j$) grading exactly as in \cite{K, ORS}, and these gradings do not depend on the data of $\XX$. The odd Khovanov differential, \[\partial: \CKh(\cP(L)) \rightarrow \CKh(\cP(L)),\] which also does not depend on the data of $\XX$, is defined exactly as in \cite{K, ORS} as a signed sum (specified by an {\em edge assignment} as in \cite[Defn. 1.1]{ORS}) of elementary merge maps, $F_M: F(I) \rightarrow F(I')$ and split maps $F_\Delta: F(I) \rightarrow F(I')$ associated to edges of the hypercube. 

For completeness, the definitions of $F_M$ and $F_\Delta$ are also briefly recalled below.

Let $I_0, I_1: \mathcal{X} \rightarrow \{0,1\}$ be two vertices for which there is an oriented edge from $I_0$ to $I_1$, as in \cite[p.3]{ORS}.\footnote{$I_1$ is sometimes called an {\em immediate successor} of $I_0$.} If two components, $a_1$ and $a_2$, of $\cP^{I_0}(L)$ merge to a single component, $a$, of $\cP^{I_1}(L)$, there is a natural identification $V(I_1) \cong V(I_0)/(a_1 - a_2)$ coming from identifying $a$ with $[a_1] = [a_2]$. The merge map
\[F_M:  \Wedge^*V(I_0) \rightarrow \Wedge^*V(I_1)\] is the map on the exterior algebra induced by the projection followed by this natural identification: $V(I_0) \rightarrow V(I_0)/(a_1-a_2) \cong V(I_1)$.

If a single component, $a$, of $\cP^{I_0}(L)$ splits into two components, $a_1$ and $a_2$, of $\cP^{I_1}(L)$ and the local arrow decorating the split region points from $a_1$ to $a_2$, then the split map is defined by
\[F_\Delta: \Wedge^*(V(I_0)) \longrightarrow \Wedge^*\left(\frac{V(I_1)}{(a_1-a_2)}\right) \longrightarrow (a_1 - a_2) \wedge \Wedge^*V(S_1) \longrightarrow \Wedge^*V(S_1),\]
where the first map is the inverse of the natural identification described in the definition of the merge map, the final map is the inclusion, and the middle map is an explicit identification of the exterior algebra of the quotient, $V(I_1)/(a_1-a_2)$, as \[(a_1-a_2) \wedge \Wedge^*V(S_1) \subset \Wedge^*V(I_1).\]
 
To obtain the annular ($k$) grading, begin by choosing an oriented arc $\gamma$ from $\XX$ to $\OO$ that misses all crossings of $\cP(L)$. It will be clear from the construction that the $k$ grading is independent of this choice.

Exactly as in the even case (cf. \cite[Sec. 4.2]{JacoFest}), we have a one-to-one correspondence between distinguished basis elements $a_S$ of $F(I) \subseteq \CKh(\cP(L))$ and {\em orientations} of the Kauffman state $\cP^I(L)$, defined as follows. Choose the {\em clockwise} (CW) orientation on a component $a_i$ of the Kauffman state if $i \in S$ and the {\em counterclockwise} (CCW) orientation on $a_i$ if $i \not\in S$. The ``$k$" grading of a distinguished basis element is now defined to be the algebraic intersection number of the corresponding oriented Kauffman state with a fixed oriented arc $\gamma$ from $\XX$ to $\OO$ that misses all crossings of $\cP(L)$. See Figure \ref{fig:grading}.

\begin{figure}
\includegraphics[height=1.5in]{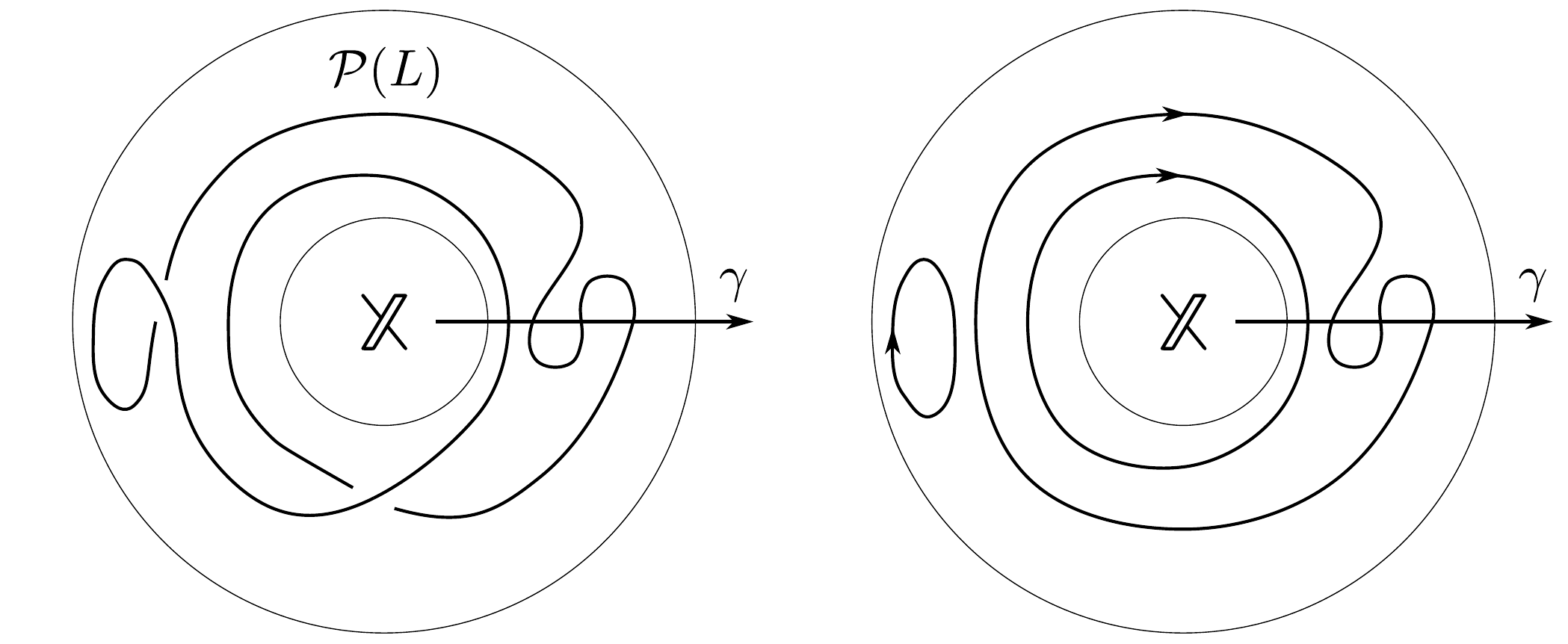}
\caption{Annular link diagram $\cP(L)$ and an oriented Kauffman state of $\cP(L)$ of ``$k$'' degree $-2$. For simplicity, the arrows at the crossings are not shown in the picture.}
\label{fig:grading}
\end{figure}

\begin{lemma} \label{lem:bicomplex}The odd Khovanov differential $\partial: \CKh(\cP(L)) \rightarrow \CKh(\cP(L))$ is non-increasing in the $k$ grading. Indeed, it can be decomposed as $\partial = \partial_0 + \partial_-$, where $\partial_0$ has $k$-degree $0$, and $\partial_{-}$ has $k$-degree $-2$.
\end{lemma}

\begin{proof}
Roberts proves (\cite[Lem. 1]{LRoberts}) that the {\em even} Khovanov differential is non-increasing in this extra grading and decomposes according to $k$--degree as described in the statement of the lemma. As noted in \cite{OS_Branched} and \cite{ORS}, the odd merge and split maps $F_M$ and $F_\Delta$ agree modulo $2$ with the merge and split maps in even Khovanov homology, so the odd Khovanov differential, $\partial$, has precisely the same decomposition according to $k$--degree. 
\end{proof}

Decomposing $\partial^2 = 0$ into its $k$--homogeneous pieces, we see that $\partial_0, \partial_-$ are two anticommuting differentials on $\CKh(L)$. The homology with respect to $\partial_0$ is triply-graded. We will denote it by: 
\[\SKh(L) := \bigoplus_{(i,j,k) \in \Z^3} \SKh^{i}(L;j,k)\] and refer to it as the {\em odd annular Khovanov homology} of $L$.

In the next subsection, we will define a $\gloneone$ action on $\SKh(L)$ and show that this action is invariant under annular Reidemeister moves, hence yields an invariant of the isotopy class of $L \subseteq A \times I$. We will also discuss the interaction of the $\gloneone$ action with the $i$, $j$, and $k$ gradings on the complex and conclude that, when regarded simply as a triply-graded vector space without a Lie superalgebra action, $\SKh(L)$ is an invariant of the isotopy class of $L \subseteq A \times I$.

\subsection{Definition and invariance of the $\mathbf{\mathfrak{gl}(1|1)}$ action on $\SKh(L)$}\label{subs:definition}
As before, we will denote by $V(I)$ the formal span of the components $a_1,\ldots,a_n$ of $\cP^I(L)$, and by $F(I)=\Wedge^*(V(I))$ the exterior algebra of $V(I)$.

We will further use the following notations: $|L|$ will denote the number of link components, $n_+$ (resp., $n_-$) will denote the number of positive (resp., negative) crossings in the link projection, and $|I|$ will denote the number of crossings $c\in\mathcal{X}$ such that $I(c)=1$. Moreover, $n_t$ (resp., $n_e$) will denote the number of trivial (resp., essential) components of $\cP^I(L)$, where a component $a_i$ is called trivial (resp., essential) if it is zero (resp., nonzero) in the first homology of $S^2-\{\XX,\OO\}$.

Using these notations, we can define the tri-degree of an element $a_{i_1}\wedge\ldots\wedge a_{i_\ell}\in F(I)$ by $i=|I|-n_-$, $j=n-2\ell+|I|+n_+-2n_-$, and $k=n_e-2\ell_e$, where $n=n_t+n_e$ denotes the number of components of $\cP^I(L)$ and $\ell_e$ denotes the number of indices $r\in\{1,\ldots,\ell\}$ for which the component $a_{i_r}$ is essential.

We will now regard $F(I)$ as a vector superspace with the supergrading given by the modulo $2$ reduction of $(j-|L|)/2$, where $j$ denotes the quantum degree just defined. It is known that $j$ always has the same parity as $|L|$, and hence the modulo $2$ reduction of $(j-|L|)/2$ is well-defined.

In what follows, we will assume that the components $a_1,\ldots,a_n$ are ordered so that $a_1,\ldots,a_{n_t}$ are trivial and $a_{n_t+1},\ldots,a_n$ are essential, and that the essential components are ordered according to their proximity to the basepoint $\XX$, so that $a_{n_t+1}$ is the essential component which is closest to $\XX$.

We can then write $F(I)$ as a tensor product of two vector superspaces
\[
F(I)=\Wedge^*(\mbox{Span}_{\C}\{a_1,\ldots,a_{n_t}\})\otimes\Wedge^*(\mbox{Span}_{\C}\{a_{n_t+1},\ldots,a_n\}),
\]
where we define the supergrading on the two tensor factors as follows. If $\ell$ denotes the natural $\Z_{\geq 0}$ degree on the exterior algebra, then the superdegree on the first tensor factor is given by the modulo $2$ reduction of $\ell+(n+|I|+n_+-2n_--|L|)/2$, and the superdegree on the second tensor factor as the modulo $2$ reduction of $\ell$.

Now equip $\mbox{Span}_{\C}\{a_{n_t+1},\ldots,a_n\}$ with the unique symmetric bilinear form $\langle-,-\rangle$ for which the vectors $a_{n_t+1},\ldots,a_n$ are orthonormal, and let $a_I,b_I\in \mbox{Span}_{\C}\{a_{n_t+1},\ldots,a_n\}$ be the vectors
\[
a_I:=a_{n_t+1}+a_{n_t+2}+\ldots+a_n\quad\mbox{and}\quad
b_I:=a_{n_t+1}-a_{n_t+2}+\ldots+(-1)^{n_e-1}a_n.
\]
Observe that $\langle a_I,b_I\rangle=m$, where
\[
m =\begin{cases}
0&\mbox{if $n_e$ is even,}\\
1&\mbox{if $n_e$ is odd.}
\end{cases}
\]
Since $n_e$ has the same parity as the winding number of $L$ around $\XX$, $m$ only depends on the parity of this winding number and not on the particular choice of $I$. Define a linear map $\rho_{U^\prime}\colon\gl(1|1)\rightarrow\gl(U^\prime)$ for $U^\prime:=\Wedge^*(\mbox{Span}_{\C}\{a_{n_t+1},\ldots,a_n\})$ by
\[
\begin{array}{rcl}
\rho_{U^\prime}(h_+)(a_{i_1}\wedge\ldots\wedge a_{i_\ell})&=&m a_{i_1}\wedge\ldots\wedge a_{i_\ell},\\
\rho_{U^\prime}(h_-)(a_{i_1}\wedge\ldots\wedge a_{i_\ell})&=&(n_e-2\ell)a_{i_1}\wedge\ldots\wedge a_{i_\ell},\\
\rho_{U^\prime}(e)(a_{i_1}\wedge\ldots\wedge a_{i_\ell})&=&(a_{i_1}\wedge\ldots\wedge a_{i_\ell})\intprodrev a_I,\\
\rho_{U^\prime}(f)(a_{i_1}\wedge\ldots\wedge a_{i_\ell})&=&a_{i_1}\wedge\ldots\wedge a_{i_\ell}\wedge b_I\\
\end{array}
\]
for all $n_t+1\leq i_1<\ldots<i_\ell\leq n$.

Comparing with Subsection~\ref{subs:exterior_action}, we see that $\rho_{U^\prime}$ endows the vector superspace $U^\prime$ with a well-defined action of $\gl(1|1)$. We can extend this action to an action on $F(I)$ by regarding $\Wedge^*(\mbox{Span}_{\C}\{a_1,\ldots,n_t\})$ as a trivial $\gl(1|1)$ representation. In other words, we can define a $\gl(1|1)$ action on $F(I)$ by setting
\[
\rho_{F(I)}(x):=\id\otimes\rho_{U^\prime}(x)
\]
for all $x\in\gl(1|1)$, where $\id$ denotes the identity map of $\Wedge^*(\mbox{Span}_{\C}\{a_1,\ldots,n_t\})$, and $\otimes$ denotes the ordinary (ungraded) tensor product of linear maps.

\begin{remark}\label{remark:gradings} The reader should note that $h_-$ acts on a vector $a_{i_1}\wedge\ldots\wedge a_{i_\ell}\in F(I)$ by scalar multiplication by $k=n_e-2\ell_e$, where $\ell_e$ denotes the number of essential components in the wedge product $a_{i_1}\wedge\ldots\wedge a_{i_\ell}$. Therefore, the $k$-grading on $F(I)$ can be viewed as the ``weight space grading'' with respect to the action of $h_-$. It is further clear that the $\gl(1|1)$ action preserves the $i$-grading because the $i$-grading is constant on $F(I)$. Moreover, the $\gl(1|1)$ action preserves the $(j-k)$-grading because the $\gl(1|1)$ action preserves the number $\ell_t$ of trivial components in a wedge product $a_{i_1}\wedge\ldots\wedge a_{i_\ell}\in F(I)$, and $j-k=n-2\ell+|I|+n_+-2n_--(n_e-2\ell_e)=n_t-2\ell_t+|I|+n_+-2n_-$.
\end{remark}

\begin{lemma}\label{lemma:intertwine} Let $I_0,I_1\colon\mathcal{X}\rightarrow\{0,1\}$ be two vertices of the resolution hypercube for which there is an oriented edge from $I_0$ to $I_1$ and let $F_W\colon F(I)\rightarrow F(I^\prime)$ for $W=M$ or $W=\Delta$ be the associated merge or split map. Then the $k$-degree preserving part of $F_W$ intertwines the actions of $\gl(1|1)$ on $F(I)$ and $F(I^\prime)$.
\end{lemma}

Although it is possible to prove Lemma~\ref{lemma:intertwine} directly, we will defer the proof to the next subsection, where we will give an alternative description of the $\gl(1|1)$ action on $F(I)$

Lemma~\ref{lemma:intertwine} tells us that the boundary maps in the odd annular Khovanov complex intertwine the $\gl(1|1)$ action, and this in turn implies that there is an induced $\gl(1|1)$ action on the odd annular Khovanov homology of an annular link diagram. By the remark preceding the lemma, it is further clear that this action preserves two gradings, namely the $i$-grading and the $(j-k)$-grading. In this sense, odd annular Khovanov homology becomes a bigraded $\gl(1|1)$ representation.

\begin{theorem}\label{theorem:invariance} If two annular link diagrams $\cP(L)$ and $\cP(L^\prime)$ differ by an annular Reidemeister move, then  $\SKh(L)$ and $\SKh(L^\prime)$ are isomorphic as bigraded $\gl(1|1)$ representations.
\end{theorem}

\begin{proof} Our proof of this theorem will closely follow the original proof of invariance of odd Khovanov homology given in \cite{ORS}. In each step of the proof, we will verify that the relevant complexes and chain maps defined in \cite{ORS} are compatible with the $\gl(1|1)$ actions on the $F(I)$ when reinterpreted in the annular setting.

{\em Invariance under annular Reidemeister moves of type I.}
Consider an annular link diagram $\mathcal{D}^\prime$ which is obtained from an annular link diagram $\mathcal{D}$ by performing a left-twist Reidemeister I move, so that $\mathcal{D}^\prime$ contains a single positive crossing which is not already present in $\mathcal{D}$. Let $\mathcal{D}_0$ and $\mathcal{D}_1$ denote the two diagrams obtained from $\mathcal{D}^\prime$ by resolving this crossing in the two possible ways. See Figure~\ref{fig:RI}. Then $\mathcal{D}_1$ is isotopic to $\mathcal{D}$, and $\mathcal{D}_0$ is isotopic to a union of $\mathcal{D}$ with a small unknotted circle.

\begin{figure}
\includegraphics[height=0.9in]{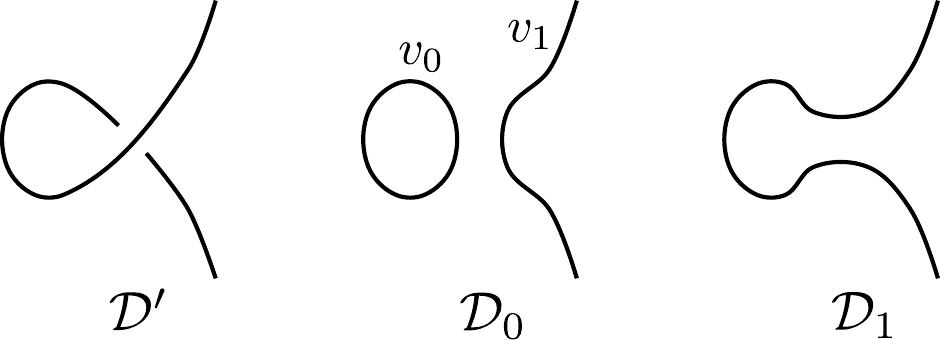}
\caption{Reidemeister I. The diagram $\mathcal{D}^\prime$ and the resolutions $\mathcal{D}_0$ and $\mathcal{D}_1$.}
\label{fig:RI}
\end{figure}

Arguing as in the proof of \cite[Prop.3.1]{ORS}, we can identify the odd annular Khovanov complex of $\mathcal{D}^\prime$ with a mapping cone of a chain map
\[
D\colon\ACKh(\mathcal{D}_0)\longrightarrow\ACKh(\mathcal{D}_1),
\]
where $\ACKh(\mathcal{D}_i)$ denotes the odd annular Khovanov complex of $\mathcal{D}_i$. The map $D$ is surjective, and hence its mapping cone is quasi-isomorphic to $\ker(D)$. Moreover, since $D$ is given by boundary maps, it intertwines the action of $\gl(1|1)$, and hence $\ker(D)$ is itself a $\gl(1|1)$ representation.

Let $v_0$ denote the small circular component of $\mathcal{D}_0$ and let $v_1$ denote the component which connects to $v_0$. Note that while $v_0$ is always trivial, $v_1$ can either be trivial or essential. In the former case, $\ker(D)$ is equal to the complex $(v_1-v_0)\wedge\ACKh(\mathcal{D}_0)$, and using that $\gl(1|1)$ acts trivially on $v_0$ and $v_1$, one can see that this complex is isomorphic (in the category of complexes of $\gl(1|1)$ representations) to the complex $\ACKh(\mathcal{D})$. Similarly, if $v_1$ is essential, then $\ker(D)$ is equal to $v_0\wedge\ACKh(\mathcal{D}_0)$, and using that $v_0$ is trivial, one can again see this complex is isomorphic to $\ACKh(\mathcal{D})$. In either case, we therefore obtain that $\ACKh(\mathcal{D}^\prime)$ is quasi-isomorphic to $\ACKh(\mathcal{D})$, proving invariance under annuular Reidemeister moves of type I.

{\em Invariance under annular Reidemeister moves of type II.}
Next, assume $\mathcal{D}^\prime$ is obtained from $\mathcal{D}$ by performing an annular Reidemeister move of type II. For $i,j\in\{0,1\}$, let $\mathcal{D}_{ij}$ denote the diagram obtained from $\mathcal{D}^\prime$ by choosing the $i$- and the $j$-resolution at the two crossings of $\mathcal{D}^\prime$ which are not present in $\mathcal{D}$. Assume the numbering of the crossings is such that $\mathcal{D}_{01}$ is isotopic to $\mathcal{D}$, and $\mathcal{D}_{10}$ is obatined from $\mathcal{D}_{00}$ by adding a small unknotted circle. See Figure~\ref{fig:RII}. Following the proof of \cite[Prop. 3.2]{ORS}, we can write the complex $\ACKh(\mathcal{D}^\prime)$ in the following form:
\[
\begin{tikzcd}
\ACKh(\mathcal{D}_{01})\arrow{r}&\ACKh(\mathcal{D}_{11})\\
\ACKh(\mathcal{D}_{00})\arrow{r}\arrow{u}&\ACKh(\mathcal{D}_{10})\arrow{u}
\end{tikzcd}
\]
In this diagram, all arrows represent maps of $\gl(1|1)$ representations. Let $v_2$ denote the small circular component in $\mathcal{D}_{10}$ and let $X\subset\ACKh(\mathcal{D}_{10})$ be the subcomplex spanned by all elements of the form $a_{i_1}\wedge\ldots\wedge a_{i_{\ell}}$ which don't contain $v_2$ as a factor. Since $v_2$ is a trivial component, $\gl(1|1)$ acts trivially on $v_2$, and hence the $\gl(1|1)$ action preserves the subcomplex $X$. Moreover, the restriction $X\rightarrow\ACKh(\mathcal{D}_{11})$ of the right vertical arrow to $X$ is an isomorphism, and thus the above complex is quasi-isomorphic to a complex of the form
\[
\begin{tikzcd}
\ACKh(\mathcal{D}_{01})&\\
\ACKh(\mathcal{D}_{00})\arrow{r}\arrow{u}&\ACKh(\mathcal{D}_{10})/X.
\end{tikzcd}
\]
In this complex, the horiziontal arrow is an isomorphism, and so the above complex is quasi-isomorphic to $\ACKh(\mathcal{D}_{01})$, which is in turn isomorphic to $\ACKh(\mathcal{D})$.

\begin{figure}
\includegraphics[height=2in]{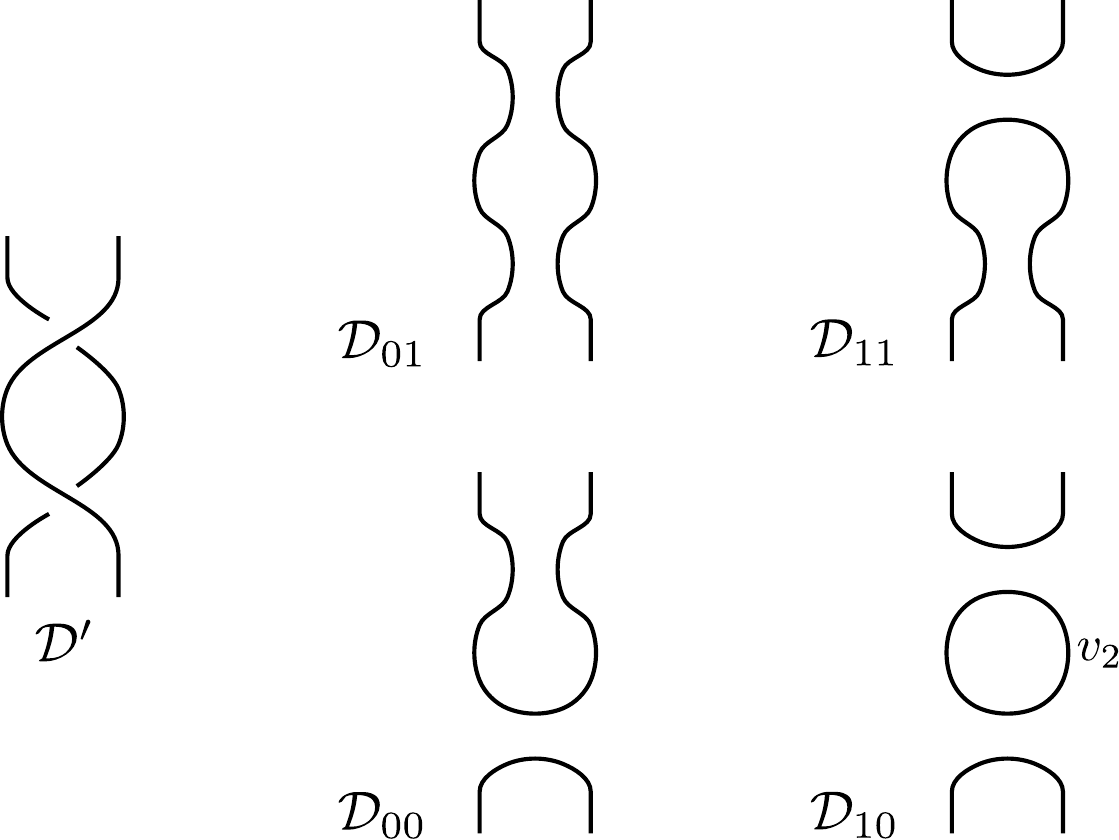}
\caption{Reidemeister II. The diagram $\mathcal{D}^\prime$ and the resolutions $\mathcal{D}_{ij}$.}
\label{fig:RII}
\end{figure}

{\em Invariance under annular Reidemeister moves of type III.} Suppose $\mathcal{D}^\prime$ is obtained from $\mathcal{D}$ by performing an annular Reidemeister III move. By repeating the arguments used in the proof of \cite[Prop. 3.3]{ORS}, one can show that $\ACKh(\mathcal{D})$ is quasi-isomorphic to a complex $L$, which fits into a short exact sequence
\[
0\longrightarrow R\stackrel{\Psi}{\longrightarrow} L\stackrel{\Phi}{\longrightarrow} P\longrightarrow 0.
\]
In \cite{ORS}, the complexes $L$ and $R$ are depicted on the left and on the right of Figure~11, and the complex $P$ is depicted in Figure~10.
Using similar arguments as above, one can see that the complexes $R$, $L$, and $P$ are complexes of $\gl(1|1)$ representations, and that $L$ is quasi-isomorphic to $\ACKh(\mathcal{D})$ in the category of $\gl(1|1)$ representations. Moreover, the maps $\Psi$ and $\Phi$ are given by scalar multiplication on the spaces associated to the vertices of (partial) resolution cubes, and since $\gl(1|1)$ acts these spaces by linear endomorphisms, it is clear that $\Psi$ and $\Phi$ intertwine the $\gl(1|1)$ actions. As in the proof of \cite[Prop. 3.3]{ORS}, the complex $R$ is acyclic, and so $L$ and hence $\ACKh(\mathcal{D})$ is quasi-isomorphic in the category of $\gl(1|1)$ representations to $P$.

By repeating the above arguments, one can show that $\ACKh(\mathcal{D}^\prime)$ is quasi-isomorphic to an analogous complex $P^\prime$, and arguing as in the proof of \cite[Prop. 3.3]{ORS}, one can see that this complex agrees with the complex $P$ up to possible signs. Following \cite{ORS}, one can then show that $P$ is isomorphic to $P^\prime$, via an isomorphism which is given by scalar multiplication on the spaces associated to vertices of (partial) resolution cubes. Since $\gl(1|1)$ acts linearly on these spaces, it follows that this isomorphism intertwines the $\gl(1|1)$ actions, and so the complexes $\ACKh(\mathcal{D})$ and $\ACKh(\mathcal{D}^\prime)$ are quasi-isomorphic in the category of $\gl(1|1)$ representations.
\end{proof}

\begin{remark} The choice of the supergrading on odd annular Khovanov homlology is not unique. In fact, by the definitions and by Remark~\ref{remark:gradings}, the boundary maps and the $\gl(1|1)$ action preserve the $(j-k)$-grading. Therefore, any shift of the supergrading by a function of $j-k$ yields a new supergrading, which is also compatible with the $\gl(1|1)$ action. For example, shifting the superdgree by the modulo $2$ reduction of $(j-k+|L|+m)/2$, where $m$ is as in the definition of the $\gl(1|1)$ action, yields a new superdegree, which is given explicitly by the modulo $2$ reduction of $(k-m)/2$.
\end{remark}

\subsection{An alternative description of the $\mathbf{\mathfrak{gl}(1|1)}$ action}\label{subs:alternative}
In this subsection, we will identify the vector superspace $F(I)$ with an isomorphic vector superspace $\tilde{F}(I)$. Using this identification, we will give a new description of the $\gl(1|1)$ action on $F(I)$ and of the boundary maps in the odd annular Khovanov complex. We will then use this description to prove a stronger version of Theorem~\ref{theorem:invariance}.

As before, we will denote by $a_1,\ldots,a_n$ the components of the resolution $\cP^I(L)$. Moreover, we will denote by $V$ a $2$-dimensional vector superspace spanned by two homogeneous elements $v_+$ and $v_-$ of superdegrees $\bar{0}$ and $\bar{1}$, respectively. Let $\tilde{F}(I)$ denote the vector superspace
\[
\tilde{F}(I):=(V\otimes\ldots\otimes V)\langle (n+|I|+n_+-2n_--|L|)/2\rangle,
\]
where there are $n$ tensor factors on the right-hand side, and where it is understood that the $i$th tensor factor corresponds to the component $a_i$. We can define an isomorphism of vector superspaces
\[
\alpha_I\colon F(I)\longrightarrow\tilde{F}(I)
\]
by sending the element $a_{i_1}\wedge\ldots\wedge a_{i_\ell}\in F(I)$ for $i_1<\ldots<i_{\ell}$ to the element $v_{\epsilon_1}\otimes\ldots\otimes v_{\epsilon_n}\in\tilde{F}(I)$ where
\[
\epsilon_i=\begin{cases}
+&\mbox{if $i\notin\{i_1,\ldots,i_{\ell}\}$,}\\
-&\mbox{if $i\in\{i_1,\ldots,i_{\ell}\}$.}
\end{cases}
\]

\begin{remark} Under this isomorphism, the $(i,j,k)$-trigrading on $F(I)$ correspond to an $(i,j,k)$-trigrading on $\tilde{F}(I)$. Explicitly, the $i$-grading is constant on $\tilde{F}(I)$ and given by $i=|I|-n_-$. The $j$-grading on $\tilde{F}(I)$ is given by \[j(v_{\epsilon_1}\otimes\ldots\otimes v_{\epsilon_n})=j(v_{\epsilon_1})+\ldots+j(v_{\epsilon_n})+|I|+n_+-2n_-,\] where the $j$-grading on the $i$th tensor factor is defined by $j(v_\pm):=\pm 1$. The $k$-grading on $\tilde{F}(I)$ is given by \[k(v_{\epsilon_1}\otimes\ldots\otimes v_{\epsilon_n})=k(v_{\epsilon_1})+\ldots+k(v_{\epsilon_n}),\] where the $k$-grading on the $i$th tensor factor is defined by $k(v_\pm):=0$ if $a_i$ is trivial, and $k(v_\pm):=\pm 1$ if $a_i$ is essential.
\end{remark}

\begin{remark} It should be noted that the map $\alpha_I$ depends nontrivially on the ordering of the components $a_1,\ldots,a_n$ of the resolution $\cP^I(L)$. However, different orderings lead to coherent maps $\alpha_I$, in the following sense. If two orderings differ by exchanging the $a_i$ and $a_{i+1}$, then the maps $\alpha_I$ and $\alpha_I^\prime$ associated to these two orderings fit into a commutative diagram
\[
\qquad
\begin{tikzcd}[row sep=small, column sep=huge]
&\tilde{F}(I) \arrow{dd}{\operatorname{id}^{\otimes(i-1)}\otimes\tau\otimes\operatorname{id}^{\otimes(n-i-1)}}\\
F(I)\arrow[start anchor = north east, end anchor = west]{ur}{\alpha_I} \arrow[start anchor = south east, end anchor = west]{dr}[swap]{\alpha_I^\prime}&\\
&\tilde{F}(I)
\end{tikzcd}
\]
where $\tau\colon V\otimes V\rightarrow V\otimes V$ denotes the twist map given by $\tau(v\otimes w)=(-1)^{|v||w|}w\otimes v$ for all homogeneous elements $v,w\in V$.
\end{remark}

\begin{remark} Let $\tilde{\alpha}_I$ denote the map $\alpha_I$ without the overall shifts of the supergrading on the domain and the codomain. That is, $\tilde{\alpha}_I$ is a map
\[
\tilde{\alpha}_I\colon \Wedge^*(\mbox{Span}_{\C}\{a_1,\ldots,a_n\})\longrightarrow V^{\otimes n},
\]
where the supergrading on the exterior algebra is defined by collapsing the natural $\Z_{\geq 0}$-grading on the exterior algebra to a $\Z_2$-grading. We can now identify the $i$th tensor factor of $V^{\otimes n}$ with $\Wedge^*(\mbox{Span}_{\C}\{a_i\})$ by using the linear map given by $v_+\mapsto 1$ and $v_-\mapsto a_i$. Under this identification, the map $\tilde{\alpha}_I$ becomes the ``obvious'' algebra isomorphism
\[
\tilde{\alpha}_I\colon\Wedge^*(\mbox{Span}_{\C}\{a_1,\ldots,a_n\})\longrightarrow\Wedge^*(\mbox{Span}_{\C}\{a_1\})\otimes\ldots\otimes\Wedge^*(\mbox{Span}_{\C}\{a_n\})
\]
given by sending the generator $a_i$ to the element $1\otimes\ldots\otimes 1\otimes a_i\otimes 1\otimes\ldots 1$. Here, $\otimes$ denotes the supergraded tensor product of supergraded algebras: it is the ordinary tensor product on the level of vector superspaces, but the algebra multiplication is given by $(a\otimes b)\cdot(c\otimes d):=(-1)^{|b||c|}(a\cdot c)\otimes(b\cdot d)$ for all homogeneous algebra elements $a,b,c,d$.
\end{remark}

We will now define a $\gl(1|1)$ action on the vector superspace $\tilde{F}(I)$. To define this action, we will first define a $\gl(1|1)$ action on each tensor factor of \[\tilde{F}(I)=V^{\otimes n}\langle (n+|I|+n_+-2n_--|L|)/2\rangle,\] and then regard $\tilde{F}(I)$ as the tensor product representation (with shifted supergrading).

As in the previous subsection, we will assume that the components of $\cP^I(L)$ are ordered so that the trivial ones precede the essential ones, and that the essential components of $\cP^I(L)$ are ordered according to their proximity to the basepoint $\XX$.

If $V$ is a tensor factor of $\tilde{F}(I)$ which corresponds to a trivial component of $\cP^I(L)$, then we now define the $\gl(1|1)$ action on $V$ to be trivial. If $V$ is a tensor factor of $\tilde{F}(I)$ which corresponds to an essential component $a_i$ with $i-n_t-1$ even, then we identify $V$ with the fundamental representation $V=L_{(1,0)}=\C^{1|1}$ of $\gl(1|1)$. Explicilty, the $\gl(1|1)$ action on such a factor is given by:
\[
\begin{tikzcd}[column sep=huge]
v_+\arrow[out=-150,in=150,loop]{l}{(h_1,h_2)=(1,0)}\arrow[bend right=40]{r}[swap]{f=1}&v_-\arrow[bend right=40]{l}[swap]{e=1}\arrow[out=30,in=-30, loop]{r}{(h_1,h_2)=(0,1)}
\end{tikzcd}
\]

Finally, if $V$ is a tensor factor of $\tilde{F}(I)$ which corresponds to an essential component $a_i$ with $i-n_t-1$ odd, then we identify $V$ with the representation $V^*\langle 1\rangle=L_{(1,0)}^*\langle 1\rangle$ via the map which takes $v_+$ to $v_-^*$ and $v_-$ to $-v_+^*$. Explicitly, the $\gl(1|1)$ action on such a factor is given by:
\[
\begin{tikzcd}[column sep=huge]
v_+\arrow[out=-150,in=150,loop]{l}{(h_1,h_2)=(0,-1)}\arrow[bend right=40]{r}[swap]{f=-1}&v_-\arrow[bend right=40]{l}[swap]{e=1}\arrow[out=30,in=-30, loop]{r}{(h_1,h_2)=(-1,0)}
\end{tikzcd}
\]

In summary, we obtain a $\gl(1|1)$ action on $\tilde{F}(I)$, and we now claim that this action corresponds to the $\gl(1|1)$ action on $F(I)$ defined in the previous subsection.

\begin{lemma}\label{lemma:intertwineactions} For each $I\colon\mathcal{X}\rightarrow\{0,1\}$, the isomorphism $\alpha_I\colon F(I)\rightarrow\tilde{F}(I)$ intertwines the $\gl(1|1)$ actions on $F(I)$ and $\tilde{F}(I)$.
\end{lemma}
\begin{proof} On both $F(I)$ and $\tilde{F}(I)$, the generator $h_+=h_1+h_2$ acts by multiplication by $m$, where $m$ is equal to $0$ if $n_e$ is even and equal to $1$ if $n_e$ is odd. Likewise, the generator $h_-=h_1-h_2$ acts on both $F(I)$ and $\tilde{F}(I)$ by multiplication by $k$, where $k$ denotes the annular grading. Thus it is clear that $\alpha_I$ intertwines the actions of $h_+$ and $h_-$ and, therefore, also the actions of $h_1$ and $h_2$.

To see that $\alpha_I$ also intertwines the actions of $e$ and $f$, we first note that $e$ and $f$ act trivially on $\Wedge^*\mbox{Span}_{\C}\{a_1,\ldots,a_{n_t}\}$ and also on the first $n_t$ tensor factors of $\tilde{F}(I)$, as these correspond to the trivial components $a_1,\ldots,a_{n_t}$.

Now suppose $a_i$ is an essential component, and suppose that the essential components are ordered according to their proximity to the basepoint $\XX$, as before. If $i-n_t-1$ is even, then the definitions of $a_I$ and $b_I$ imply that $e$ and $f$ act on $\Wedge^*\mbox{Span}_{\C}\{a_i\}$ by the maps $v\mapsto v\intprodrev a_i$ and $v\mapsto v\wedge a_i$, respectively. Explicitly,
\[
\begin{tikzcd}[column sep=huge]
1\arrow[bend right=40]{r}[swap]{f=1}&a_i\arrow[bend right=40]{l}[swap]{e=1}
\end{tikzcd}
\]
and this diagram is consistent with the diagram for the $\gl(1|1)$ action on the $i$th tensor factor of $\tilde{F}(I)$.

Likewise, if $i-n_t-1$ is odd, then $e$ and $f$ act on $\Wedge^*\mbox{Span}_{\C}\{a_i\}$ by the maps $v\mapsto v\intprodrev a_i$ and $v\mapsto v\wedge (-a_i)$, or more explicitly,
\[
\begin{tikzcd}[column sep=huge]
1\arrow[bend right=40]{r}[swap]{f=-1}&a_i\arrow[bend right=40]{l}[swap]{e=1}
\end{tikzcd}
\]
and again this diagram is consistent with the diagram for the $\gl(1|1)$ action on the $i$th tensor factor of $\tilde{F}(I)$.

To complete the proof, we note that if the maps $v\mapsto v\intprodrev a_i$ and $v\mapsto v\wedge (\pm a_i)$ act on a wedge product of the form
\[
a_{i_1}\wedge\ldots\wedge a_{i_r}\wedge\Wedge^*\mbox{Span}_{\C}\{a_i\}\wedge a_{i_s}\wedge\ldots\wedge a_{i_\ell}
\]
for $i_1<\ldots<i_r<i<i_s<\ldots<i_{\ell}$, then they pick up the sign $(-1)^{\ell-s+1}$, where the sign comes from permuting $\pm a_i$ across $a_{i_s}\wedge\ldots\wedge a_{i_\ell}$. Because of our definition of the tensor product of $\gl(1|1)$ representations, the same sign occurs when $e$ and $f$ act on the tensor product
\[
v_{\epsilon_1}\otimes\ldots\otimes v_{\epsilon_{i-1}}\otimes V\otimes v_{\epsilon_{i+1}}\otimes\ldots\otimes v_{\epsilon_n},
\]
where here $\epsilon_1,\ldots,\epsilon_{i-1},\epsilon_{i+1},\ldots,\epsilon_n$ are related to $i_1,\ldots,i_r,i_s,\ldots,i_\ell$ as in the definition of the map $\alpha_I$.
\end{proof}

The next lemma will relate the merge and split maps $F_M$ and $F_\Delta$ to the maps $m\colon V\otimes V\rightarrow V$ and $\Delta\colon V\rightarrow V\otimes V$ defined as follows:
\[
\begin{array}{rcl}
m&=&\begin{cases}
v_+\otimes v_+\longmapsto v_+,&v_+\otimes v_-\longmapsto v_-,\\
v_-\otimes v_-\longmapsto 0,&v_-\otimes v_+\longmapsto v_-,
\end{cases}\\[0.2in]
\Delta&=&\begin{cases}
v_+\longmapsto v_-\otimes v_+-v_+\otimes v_-,\\
v_-\longmapsto v_-\otimes v_-.
\end{cases}
\end{array}
\]
Note that $m$ is homogeneous of superdegree $\bar{0}$ and $\Delta$ is homogeneous of superdegree $\bar{1}$.

\begin{lemma}\label{lemma:intertwineboundary} Let $I_0,I_1\colon\mathcal{X}\rightarrow\{0,1\}$ be two vertices of the resolution hypercube for which there is an oriented edge from $I_0$ to $I_1$. Denote by $a_1,\ldots,a_n$ the components of $\cP^{I_0}(L)$ and by $a_1^\prime,\ldots,a_{n\mp 1}^\prime$ the components of $\cP^{I_1}(L)$, and let $\otimes$ denote the tensor product of homogeneous linear maps defined in Subsection~\ref{subs:representations}.
\begin{enumerate}
\item
If the components $a_i$ and $a_{i+1}$ merge into the component $a_i^\prime$, and if $a_j=a_j^\prime$ for $j<i$ and $a_j=a_{j-1}^\prime$ for $j>i+1$, then
\[
\alpha_{I_1}\circ F_M\circ\alpha_{I_0}^{-1}=\id^{\otimes(i-1)}\otimes m\otimes\id^{\otimes(n-i-1)}.
\]
\item If the component $a_i$ splits into the components $a_i^\prime$ and $a_{i+1}^\prime$, and if $a_j=a_j^\prime$ for $j<i$ and $a_j=a_{j+1}^\prime$ for $j>i$, then
\[
\alpha_{I_1}\circ F_\Delta\circ\alpha_{I_0}^{-1}=\id^{\otimes(i-1)}\otimes \Delta\otimes\id^{\otimes(n-i)},
\]
where it is assumed that the arrow decorating the split region points from $a_i^\prime$ to $a_{i+1}^\prime$.
\end{enumerate}
\end{lemma}

\begin{proof} Recalling the definitions of $F_M$, $F_\Delta$, and $\otimes$, it is easy to see that we can reduce to case where one of $\cP^{I_0}(L)$ and $\cP^{I_1}(L)$ has exactly two components and the other one has exactly one component.

Suppose first that $\cP^{I_0}(L)$ has exactly two components $a_1$ and $a_2$, and that $\cP^{I_1}(L)$ has a single component $a_1^\prime$. Then the map $F_M\colon F(I_0)\rightarrow F(I_1)$ is given by sending each of the components $a_1$ and $a_2$ to $a_1^\prime$. Explicitly,
\[
F_M=\begin{cases}
1\longmapsto 1,&a_1\longmapsto a_1^\prime,\\
a_1\wedge a_2\longmapsto 0,&a_2\longmapsto a_1^\prime,
\end{cases}
\]
and comparing with the definition of $m$, we see that $\alpha_{I_1}\circ F_M\circ\alpha_{I_0}^{-1}=m$.

Now suppose that $\cP^{I_0}(L)$ has a single component $a_1$, and that $\cP^{I_1}(L)$ has exactly two components $a_1^\prime$ and $a_2^\prime$. Then the map $F_\Delta\colon F(I_0)\rightarrow F(I_1)$ is given by first sending $a_1$ to either $a_1^\prime$ or $a_2^\prime$, and then wedge-multiplying the result from the left by $a_1^\prime-a_2^\prime$ (where we assume that the arrow decorating the split region points from $a_1^\prime$ to $a_2^\prime$). Explicitly,
\[
F_\Delta=\begin{cases}
1\longmapsto a_1^\prime-a_2^\prime,\\
a_1\longmapsto a_1^\prime\wedge a_2^\prime,
\end{cases}
\]
and comparing with the definition of $\Delta$, we see that $\alpha_{I_1}\circ F_\Delta\circ\alpha_{I_0}^{-1}=\Delta$.
\end{proof}

We may now regard $m$ and $\Delta$ as maps associated to merge and split operations between resolutions on an annulus. Since, up to a sign in the definition of $\Delta(v_+)$, $m$ and $\Delta$ coincide with Khovanov's multiplication and comultiplication maps, it follows that they exhibit the same behavior with respect to the $k$-grading as the latter maps. In particular, it follows from \cite{LRoberts} that there are decompositions
\[
m=m_0+m_-\quad\mbox{and}\quad\Delta=\Delta_0+\Delta_-,
\]
where $m_0$ and $\Delta_0$ have $k$-degree $0$ and $m_-$ and $\Delta_-$ have $k$-degree $-2$. The explicit form of these decompositions depends on whether the components involved in the merge or split operation are trivial or essential.

If all involved components are trivial, then $m_0=m$ and $\Delta_0=\Delta$.

If a trivial component and an essential component are merged into a single essential component, then the map $v\mapsto m_0(v_+\otimes v)$ is the identity map and the map $v\mapsto m_0(v_-\otimes v)$ is the zero map, where here it is assumed that the first factor in $V\otimes V$ corresponds to the trivial component. If the second factor corresponds to the trivial component, then the same holds true, but with $m_0$ replaced by $m_0\circ\tau$.

If a single essential component is split into a trivial component and an essential component, then the map $\Delta_0$ is given by $v\mapsto v_-\otimes v$, where here it is assumed that the first factor of $V\otimes V$ corresponds to the trivial component. If the second factor corresponds to the trivial component, then the same holds true, but with $\Delta_0$ replaced by $-\tau\circ\Delta_0$.

Finally, if two essential components are merged into a single trivial component, or if a single trivial component is split into two essential components, then the corresponding maps $m_0$ and $\Delta_0$ are given as follows:
\[
\begin{array}{rcl}
m_0&=&\begin{cases}
v_+\otimes v_+\longmapsto 0,&v_+\otimes v_-\longmapsto v_-,\\
v_-\otimes v_-\longmapsto 0,&v_-\otimes v_+\longmapsto v_-,
\end{cases}\\[0.2in]
\Delta_0&=&\begin{cases}
v_+\longmapsto v_-\otimes v_+-v_+\otimes v_-,\\  
v_-\longmapsto 0.
\end{cases}
\end{array}
\]

We can now prove Lemma~\ref{lemma:intertwine} form the previous subsection.

\begin{proof}[Proof of Lemma~\ref{lemma:intertwine}]
In view of Lemmas~\ref{lemma:intertwineactions} and \ref{lemma:intertwineboundary}, it is enough to show that the maps \[\id^{\otimes i-1}\otimes m_0\otimes\id^{\otimes(n-i-1)}\quad\mbox{and}\quad\id^{\otimes (i-1)}\otimes\Delta_0\otimes\id^{\otimes (n-i)}\] intertwine the $\gl(1|1)$ actions, where here the notation is to be understood as in the two parts of Lemma~\ref{lemma:intertwineactions}. Since $\id$ intertwines the $\gl(1|1)$ actions and because of Lemma~\ref{lemma:tensor}, it further suffices to show that $m_0$ and $\Delta_0$ intertwine the $\gl(1|1)$ actions when viewed as factors of these maps.

If $m_0$ and $\Delta_0$ correspond to merges and splits which only involve trivial components, then this is obvious becasue in this case the $\gl(1|1)$ actions are trivial on the domain and the codomain of $m_0$ and $\Delta_0$.

If $m_0$ corresponds to a merge of a trivial component with an essential component, then $m_0$ (or $m_0\circ\tau$) can be described in terms of the identity map or the zero map, depending on whether the trivial component is labeled by $v_+$ or $v_-$, and these maps, too, intertwine the $\gl(1|1)$ action. (Note that if $m_0\circ\tau$ intertwines the $\gl(1|1)$ actions, then so does $m_0$, because $\tau$ is an isomorphism of $\gl(1|1)$ representations).

If further $\Delta_0$ corresponds to splitting an essential component into a trivial and an essential component, then $\Delta_0$ (or $-\tau\circ\Delta_0$) is given by $v\mapsto v_-\otimes v$, and again this map intertwines the $\gl(1|1)$ action.

Finally, if $m_0$ or $\Delta_0$ corresponds to merging to essential components into a single trivial component, or to splitting a single trivial component into two essential components, then $m_0$ and $\Delta_0$ agree with the maps $\tilde{p}$ and $\tilde{i}$ defined in Subsection~\ref{subs:productaction}, and we have already seen that these maps intertwine the $\gl(1|1)$ action.\end{proof}

Using the alternative description of the $\gl(1|1)$ action in terms of the spaces $\tilde{F}(I)$, we can now prove the following stronger version of Theorem~\ref{theorem:invariance}.

\begin{theorem} \label{thm:homequiv} If two annular link diagrams $\cP(L)$ and $\cP(L^\prime)$ differ by an annular Reidemeister move, then the odd annular Khovanov complexes $\ACKh(\cP(L))$ and $\ACKh(\cP(L^\prime))$ are homotopy equivalent as complexes of $\gl(1|1)$ representations.
\end{theorem}

\begin{proof} In \cite{PutyraChrono}, Putyra defines a category $\cobchrono$, which generalizes Bar-Natan's category $\cobe$ defined in \cite{MR2174270}. To each decorated link diagram $\cP(L)$, Putyra assigns a chain complex $Kh(\cP(L))$, which lives in the additive closure of $\cobchrono$, and which can be viewed as a generalization of Bar-Natan's formal Khovanov bracket of $\cP(L)$.

Putyra shows that if two link diagrams $\cP(L)$ and $\cP(L^\prime)$ differ by a Reidemeister move, then there is a homotopy equivalence between the generalized Khovanov brackets $Kh(\cP(L))$ and $Kh(\cP(L^\prime))$. Moreover, Putyra constructs a functor, called a chronological TQFT, which takes $Kh(\cP(L))$ to the odd Khovanov complex of $\cP(L)$.

Putyra's construction can be carried out equally well in the annular setting. In particular, if $\cP(L)$ is an annular link diagram, then one can associate an annular version of the generalized Khovanov bracket, which lives in the additive closure of an annular version of the category $\cobchrono$. Moreover, there is a chronological TQFT defined on this annular category, which takes the generalized annular Khovanov bracket of $\cP(L)$ to the odd annular Khovanov complex $\ACKh(\cP(L))$. We will henceforth denote this chronological TQFT by $\mathcal{F}^{ann}_o$. Explicitly, $\mathcal{F}^{ann}_o$ is an additive functor given by assigning the maps $m_0\colon V\otimes V\rightarrow V$ and $\Delta_0\colon V\rightarrow V\otimes V$ to (suitably decorated) annular saddle cobordisms, and the maps $\iota\colon\C\rightarrow V$ and $\epsilon\colon V\rightarrow\C$ given by $\iota(1)=v_+$, $\epsilon(v_+)=0$, and $\epsilon(v_-)=1$ to (suitably decorated) annular cup and cap cobordisms.

We have already seen in the proof of Lemma~\ref{lemma:intertwine} that the maps $m_0$ and $\Delta_0$ intertwine the $\gl(1|1)$ action. Since annular cup and cap cobordisms can only create or annihilate trivial components, and since the $\gl(1|1)$ action is trivial on tensor factors corresponding to such components, it is further clear that the maps $\iota$ and $\epsilon$ also intertwine the $\gl(1|1)$ action. We can thus view the functor $\mathcal{F}^{ann}_o$ as a functor with values in the representation category of $\gl(1|1)$. Since this functor is also additive, it takes the homotopy equivalences that Putyra associates to Reidemeister moves to homotopy categories in the representation category of $\gl(1|1)$, and this proves the theorem.
\end{proof}

\begin{remark} The category $\cobchrono$ defined in \cite{PutyraChrono} comes equipped with a $\Z\times\Z$-grading on its morphism sets. The modulo $2$ reduction of the second $\Z$-factor in this $\Z\times\Z$-grading corresponds to the supergrading used in our definition of the $\gl(1|1)$ action on odd annular Khovanov homology.
\end{remark}

\bibliography{oddannular_gl11_arxiv}

\def\polhk#1{\setbox0=\hbox{#1}{\ooalign{\hidewidth
  \lower1.5ex\hbox{`}\hidewidth\crcr\unhbox0}}}
\begin{thebibliography}{10}

\bibitem{APS}
M.~M. Asaeda, J.~H. Przytycki, and A.~S. Sikora.
\newblock Categorification of the {K}auffman bracket skein module of
  {$I$}-bundles over surfaces.
\newblock {\em Algebr. Geom. Topol.}, 4:1177--1210 (electronic), 2004.

\bibitem{HochHom}
D.~Auroux, J.~E. Grigsby, and S.~M. Wehrli.
\newblock Sutured {K}hovanov homology, {H}ochschild homology, and the
  {O}zsv\'ath-{S}zab\'o spectral sequence.
\newblock {\em Trans. Amer. Math. Soc.}, 367(10):7103--7131, 2015.

\bibitem{MR2174270}
D.~Bar-Natan.
\newblock Khovanov's homology for tangles and cobordisms.
\newblock {\em Geom. Topol.}, 9:1443--1499 (electronic), 2005.

\bibitem{Be}
S.~Beier.
\newblock An integral lift starting in odd {K}hovanov homology of {S}zab{\'o}'s
  spectral sequence.
\newblock math.GT/1205.2256, 2012.

\bibitem{Beliakova-Putyra-Wehrli}
A.~Beliakova, K.~Putyra, and S.~M. Wehrli.
\newblock Quantum {L}ink {H}omology via {T}race {F}unctor {I}.
\newblock arXiv:1605.03523, 2016.

\bibitem{BM}
G.~Benkart and D.~Moon.
\newblock Planar rook algebras and tensor representations of
  {$\mathfrak{gl}(1|1)$}.
\newblock {\em Comm. Algebra}, 41(7):2405--2416, 2013.

\bibitem{BrundanEllis}
J.~Brundan and A.~P. Ellis.
\newblock Monoidal {S}upercategories.
\newblock {\em Comm. Math. Phys.}, 351:1045--1089, 2017.

\bibitem{epv}
A.~P. Ellis, I.~Petkova, and V.~V{\'e}rtesi.
\newblock Quantum $gl(1|1)$ and tangle {F}loer homology.
\newblock math.GT/1510.03483, 2015.

\bibitem{SchurWeyl}
J.~E. Grigsby, A.~M. Licata, and S.~M. Wehrli.
\newblock Annular {K}hovanov homology and knotted {S}chur-{W}eyl
  representations.
\newblock {\em Compositio Mathematica}, 154(3):459--502, 2018.

\bibitem{AnnularLinks}
J.~E. Grigsby and S.~M. Wehrli.
\newblock Khovanov homology, sutured {F}loer homology and annular links.
\newblock {\em Algebr. Geom. Topol.}, 10(4):2009--2039, 2010.

\bibitem{JacoFest}
J.~E. Grigsby and S.~M. Wehrli.
\newblock On gradings in {K}hovanov homology and sutured {F}loer homology.
\newblock In {\em Topology and geometry in dimension three}, volume 560 of {\em
  Contemp. Math.}, pages 111--128. Amer. Math. Soc., Providence, RI, 2011.

\bibitem{dt}
J.E. Grigsby, A.~M. Licata, and S.~M. Wehrli.
\newblock Annular {K}hovanov-{L}ee homology, braids, and cobordisms.
\newblock math.GT/1612.05953, 2016.

\bibitem{K}
M.~Khovanov.
\newblock A categorification of the {J}ones polynomial.
\newblock {\em Duke Math. J.}, 101(3):359--426, 2000.

\bibitem{ManionDecat}
A.~Manion.
\newblock On the decategorification of {O}szv{\'a}th and szab{\'o}'s bordered
  theory for knot {F}loer homology.
\newblock math.GT/1611.08001, 2016.

\bibitem{ORS}
P.~Ozsv{\'a}th, J.~Rasmussen, and Z.~Szab{\'o}.
\newblock Odd {K}hovanov homology.
\newblock {\em Algebr. Geom. Topol.}, 13:1465--1488, 2013.

\bibitem{OS_Branched}
P.~Ozsv{\'a}th and Z.~Szab{\'o}.
\newblock On the {H}eegaard {F}loer homology of branched double-covers.
\newblock {\em Adv. Math.}, 194(1):1--33, 2005.

\bibitem{oszKauffman}
P.~Ozsv{\'a}th and Z.~Szab{\'o}.
\newblock Kauffman states, bordered algebras, and a bigraded knot invariant.
\newblock {\em Adv. Math.}, 328:1088--1198, 2018.

\bibitem{pv}
I.~Petkova and V.~V{\'e}rtesi.
\newblock Combinatorial tangle {F}loer homology.
\newblock math.GT/1410.2161, 2014.

\bibitem{pv1}
I.~Petkova and V.~V{\'e}rtesi.
\newblock A self-pairing theorem for tangle {F}loer homology.
\newblock {\em Algebr. Geom. Topol.}, 16(4):2127--2141, 2016.

\bibitem{Plam}
O.~Plamenevskaya.
\newblock Transverse knots and {K}hovanov homology.
\newblock {\em Math. Res. Lett.}, 13(4):571--586, 2006.

\bibitem{PutyraChrono}
K.~Putyra.
\newblock A $2$-category of chronological cobordisms and odd {K}hovanov
  homology.
\newblock {\em Banach Center Publications}, 103:291--355, 2014.

\bibitem{QR}
H.~Queffelec and D.~E.~V. Rose.
\newblock Sutured annular {K}hovanov-{R}ozansky homology.
\newblock {\em Trans. Amer. Math. Soc.}, 370(2):1285--1319, 2018.

\bibitem{Rasmussen_Slice}
J.~Rasmussen.
\newblock Khovanov homology and the slice genus.
\newblock {\em Invent. Math.}, 182(2):419--447, 2010.

\bibitem{ReshTur}
N.~Reshetikhin and V.~G. Turaev.
\newblock Invariants of {$3$}-manifolds via link polynomials and quantum
  groups.
\newblock {\em Invent. Math.}, 103(3):547--597, 1991.

\bibitem{LRoberts}
L.~P. Roberts.
\newblock On knot {F}loer homology in double branched covers.
\newblock {\em Geom. Topol.}, 17(1):413--467, 2013.

\bibitem{Russell}
H.~Russell.
\newblock The {B}ar-{N}atan skein module of the solid torus and the homology of
  {$(n,n)$} {S}pringer varieties.
\newblock {\em Geometriae Dedicata}, 142:71--89, 2009.

\bibitem{Sar}
A.~Sartori.
\newblock The {A}lexander polynomial as quantum invariant of links.
\newblock {\em Ark. Mat.}, 53(1):177--202, 2015.

\bibitem{Sz}
Z.~Szab{\'o}.
\newblock A geometric spectral sequence in {K}hovanov homology.
\newblock {\em Topology}, 8(4):1017--1044, 2015.

\bibitem{Witten}
E.~Witten.
\newblock Quantum field theory and the {J}ones polynomial.
\newblock {\em Comm. Math. Phys.}, 121(3):351--399, 1989.

\end{thebibliography}

\end{document}